\documentclass[10pt,reqno]{amsart}

\usepackage{amssymb}
\usepackage{amscd}
\usepackage{amsfonts}
\usepackage{amsmath}
\usepackage{setspace}
\usepackage{version}

\usepackage{graphicx}

\newtheorem{theorem}{Theorem}
\newtheorem{lemma}[theorem]{Lemma}

\newtheorem{corollary}[theorem]{Corollary}

\theoremstyle{definition}

\theoremstyle{remark}

\renewcommand{\leq}{\leqslant}
\renewcommand{\geq}{\geqslant}
\newcommand\SL{\operatorname{SL}}

\newcommand\Lip{\operatorname{Lip}}
\newcommand\str{\operatorname{str}}

\newcommand\sml{\operatorname{sml}}
\newcommand\unf{\operatorname{unf}}

\newcommand\smth{\operatorname{smth}}
\newcommand\irrat{\operatorname{irrat}}
\newcommand\rat{\operatorname{rat}}

\def\E{\mathbf{E}}
\def\cE{\mathcal{E}}

\def\cF{\mathcal{F}}
\def\R{\mathbf{R}}
\def\C{\mathbf{C}}
\def\Z{\mathbf{Z}}

\def\N{\mathbf{N}}
\def\T{\mathbf{T}}
\def\B{\mathcal{B}}
\def\eps{\varepsilon}

\renewcommand{\(}{\left(}
\renewcommand{\)}{\right)}
\def\eps{\varepsilon}

\begin{document}

\title{The abelian arithmetic regularity lemma}



\author{Sean Eberhard}
\email{eberhard.math@gmail.com}

\begin{abstract}
We introduce and prove the arithmetic regularity lemma of Green and Tao in the abelian case. This exposition may serve as an introduction to the general case.
\end{abstract}

\maketitle

The purpose of this note is to provide a brief, self-contained exposition and proof of the arithmetic regularity lemma of Green and Tao~\cite{areg} in the abelian ($s=1$) case, particularly in order to aid the reader of~\cite{EGM}, or to serve as an introduction to the general ($s>1$) case. All the results of this paper can therefore be read out of the more general results expounded in the first two sections of~\cite{areg}, but to do so would require digesting the higher-order theory as well, which is a little more involved. In particular, while in~\cite{areg} the authors rely on the inverse theorem for the $U^{s+1}$ norm, we need only the inverse theorem for the $U^2$ norm, which is elementary both to state and to prove.

The results of this paper are also contained in~\cite{higher-fourier}. Compared to that treatment, we use slightly different language in a few places, and we absorb the Ratner-type theory into the statement of the regularity lemma.

The arithmetic regularity lemma states, roughly speaking, that an arbitrary function $f:[N]\to[0,1]$ is the sum of a structured part $f_{\str}$, a small part $f_{\sml}$, and a Gowers-uniform part $f_{\unf}$. Moreover we can buy higher-order uniformity of $f_{\unf}$ at the cost of more involved structure of $f_{\str}$, but here we will only be able to afford $U^2$ uniformity.

We start with the inverse theorem for the $U^2$ norm. We define the $U^2(\Z/M\Z)$ norm of a function $f:\Z/M\Z\to\C$ as
\[
  \|f\|_{U^2(\Z/M\Z)} = \(\E_{a,h_1,h_2\in \Z/M\Z} f(a)\overline{f(a+h_1)} \overline{f(a+h_2)} f(a+h_1+h_2)\)^\frac{1}{4},
\]
and then the $U^2([N])$ norm of a function $f:[N]\to\C$ as
\[
  \|f\|_{U^2([N])} = \frac{\|f\|_{U^2(\Z/M\Z)}}{\| 1_{[N]}\|_{U^2(\Z/M\Z)}},
\]
where $M\geq 2N$ and we define $f(x)=0$ if $x\notin[N]$: one easily checks that this definition is independent of the choice of $M$. We will often abbreviate $U^2([N])$ to $U^2$ when no confusion can arise.

Given $f:\Z/M\Z\to\C$ we define the Fourier transform $\hat{f}$ of $f$ by
\[
  \hat{f}(r) = \E_{x\in\Z/M\Z} f(x) e_M(-rx)
\]
for $r\in\Z/M\Z$, where $e_M(x) = e(x/M)$. The Fourier inversion formula then states
\[
  f(x) = \sum_{r\in\Z/M\Z} \hat{f}(r) e_M(rx).
\]
Using these formulae one easily proves
\[
  \|f\|_{U^2(\Z/M\Z)} = \(\sum_{r\in \Z/M\Z} |\hat{f}(r)|^4\)^\frac{1}{4}.
\]

\begin{lemma}[Inverse theorem for the $U^2$ norm]
If $f:[N]\to[-1,1]$ is a function such that $\|f\|_{U^2} \geq \delta$, then there exists $\theta\in\T$ such that
\[
  \left| \E_{n\in[N]} f(n) e(-\theta n) \right| \gg_\delta 1.
\]
\end{lemma}
\begin{proof}
The condition $\|f\|_{U^2([N])}\geq\delta$ implies that $\|f\|_{U^2(\Z/M\Z)} \gg\delta$, where $M=2N$ and as usual we extend $f$ by zero to the rest of $\Z/M\Z$. We therefore have
\[
  \sum_{r\in\Z/M\Z} |\hat{f}(r)|^4 \gg \delta^4.
\]
From Parseval's theorem and the hypothesis $|f|\leq 1$ it then follows that
\[
  \delta^4 \ll \sup|\hat{f}|^2 \(\sum_{r\in\Z/M\Z}|\hat{f}(r)|^2\) = \sup|\hat{f}|^2 \(\E_{x\in\Z/M\Z} |f(x)|^2\) \leq \sup|\hat{f}|^2.
\]
Thus $|\hat{f}(r)|\gg\delta^2$ for at least one $r\in\Z/M\Z$, so we may take $\theta = r/M$.
\end{proof}

We need a slightly modified form of the above lemma in order to apply an energy increment argument, but first we need some language. Let us say that $f:[N]\to\R$ has \emph{$1$-complexity} at most $M$ if $f(n)=F(\theta n)$ for some $F:\T^d\to\R$ and $\theta\in\T^d$ such that $d,\|F\|_{\Lip}\leq M$. Here we take the Euclidean metric
\[
  d(x,y) = \min_{z\in\Z^d} \|x-y-z\|_2
\]
on $\T^d$, and we define the \emph{Lipschitz norm} $\|F\|_{\Lip}$ of $F:\T^d\to\R$ by
\[
  \|F\|_{\Lip} = \sup_{x}|F(x)| + \sup_{x\neq y}\frac{|F(x)-|F(y)|}{d(x,y)}.
\]
The Fourier inversion formula shows that every $f:[N]\to\C$ has finite $1$-complexity, but functions of bounded $1$-complexity are special.

Our results from now on will be quantified by an arbitrary \emph{growth function}, by which we mean simply an increasing function $\cF:\R^+\to\R^+$. By $\cF\ll_X 1$ we will mean that $\cF$ is bounded by a function $\R^+\to\R^+$ depending only on the parameter $X$; in other words $\cF\ll_X 1$ means $\cF(M)\ll_{X,M} 1$.

We say $f$ is \emph{$1$-measurable} with growth $\cF$ if for every $M>0$ there is some function $f_{\str}:[N]\to\R$ of $1$-complexity at most $\cF(M)$ such that
\[
  \|f-f_{\str}\|_2\leq \frac{1}{M},
\]
where the $L^2([N])$ norm of a function $f:[N]\to\C$ is defined by
\[
  \|f\|_2 = \(\E_{x\in[N]} |f(x)|^2\)^\frac{1}{2}.
\]
A set $E\subset[N]$ is called $1$-measurable with growth $\cF$ if $1_E$ is so. Note that if $f$ and $g$ are $1$-measurable with growth $\cF$ then $f+g$ and $fg$ are $1$-measurable with growth $\ll_\cF 1$, so if $E$ and $F$ are $1$-measurable with growth $\cF$ then $E\cup F$, $E\cap F$, $E\backslash F$, and so on, are all $1$-measurable with growth $\ll_\cF 1$.

\begin{lemma}[$U^2$ inverse theorem, alternative formulation]
If $f:[N]\to[-1,1]$ is a function such that $\|f\|_{U^2} \geq \delta$, then there is a $1$-measurable set $E\subset[N]$ with growth $\ll_\delta 1$ such that
\[
  \left|\E_{n\in[N]} f(n) 1_E(n)\right|\gg_\delta 1.
\]
\end{lemma}
\begin{proof}
By the previous lemma there is some $\theta\in\T$ such that $\phi(n) = e(-\theta n)$ satisfies
\[
  \left| \E_{n\in[N]} f(n) \phi(n) \right| \gg_\delta 1.
\]
Now by replacing $\phi$ with its the real or imaginary part, and then with its positive or negative part, we may assume that $\phi$ is real and nonnegative (e.g., if we take the real and then positive parts, then $\phi(n) = (\Re\,e(-\theta n))^+$).

For $0\leq t\leq 1$, let
\[
  E_t = \{n\in[N] : \phi(n) \geq t\}.
\]
Noting that
\[
  \phi(n) = \int_0^1 1_{E_t}(n)\, dt,
\]
it follows that
\[
  \int_0^1 \left|\E_{n\in[N]} f(n) 1_{E_t}(n)\right|\,dt\gg_\delta 1,
\]
and so
\[
  \left|\E_{n\in[N]} f(n) 1_{E_t}(n)\right|\gg_\delta 1
\]
for all $t$ in a set $\Omega\subset[0,1]$ of measure $|\Omega|\gg_\delta 1$.

Among these sets $E_t$ with $t\in\Omega$ there must be some $E_t$ which is approximately invariant under small changes in $t$. Indeed, if 
\[
  M(t) = \sup_{r>0}\frac{1}{2r}\frac{1}{N}|\{n\in[N] : |\phi(n)-t|\leq r\}|
\]
then the Hardy--Littlewood maximal inequality (see any standard reference, such as~\cite{rudin-real-complex}) states
\[
  |\{t\in[0,1] : M(t)\geq \lambda\}| \ll \frac{1}{\lambda}.
\]
Since $|\Omega|\gg_\delta 1$ there is some $t\in\Omega$ such that $M(t)\ll_\delta 1$.

For any such $t$, $E_t$ is $1$-measurable with growth $\ll_\delta 1$. Indeed, note for any $r>0$ that
\[
  |\{n\in[N]: |\phi(n) - t|\leq r\}| \ll_\delta rN.
\]
Choosing $\eta:\R\to\R^+$ of Lipschitz norm $\|\eta\|_{\Lip}\ll 1/r$ such that $\eta(x)=0$ if $x<t - r$ and $\eta(x)=1$ if $x>t+r$, it follows that $\|1_{E_t} - \eta\circ\phi\|_2\ll_\delta \sqrt{r}$. Since $\phi$ is a function of $\theta n$ of Lipschitz norm $\ll 1$, this implies that $1_{E_t}$ is $1$-measurable with growth $\ll_\delta 1$.
\end{proof}

A \emph{factor} $\B$ of $[N]$ is a subalgebra of $2^{[N]}$, or equivalently a partition of $[N]$ into cells. We say a factor $\B^\prime$ \emph{refines} another $\B$ if every cell of $\B$ is a union of cells of $\B^\prime$. We call $\B$ a \emph{$1$-factor} with complexity at most $M$ and growth $\cF$ if $\B$ has $M$ cells, each of which is $1$-measurable with growth $\cF$. Note in this case that every $\B$-measurable (i.e., constant on each cell of $\B$) function $f:[N]\to[-1,1]$ is $1$-measurable with growth $\ll_{M,\cF} 1$.

For $x\in[N]$ we define $\B(x)$ to be the unique cell containing $x$, and we define the \emph{conditional expectation} $\E(f|\B)$ of a function $f:[N]\to\C$ by
\[
  \E(f|\B)(x) = \frac{1}{|\B(x)|} \sum_{y\in\B(x)} f(y).
\]
Equivalently, the function $\E(f|\B)$ is the orthogonal projection of $f$ onto the subspace of $\B$-measurable functions. Finally, with respect to a fixed function $f:[N]\to\C$, the \emph{energy} of $\B$ is $\cE(\B) = \|\E(f|\B)\|_2^2$.

\begin{corollary}[Lack of uniformity allows energy increment]
Suppose $\B$ is a $1$-factor of complexity $\leq M$ and growth $\cF$ and $f:[N]\to[-1,1]$ is a function such that $\|f-\E(f|\B)\|_{U^2([N])} \geq \delta$. Then there exists a refinement $\B^\prime$ of $\B$ of complexity $\leq 2M$ and growth $\ll_{M,\delta,\cF} 1$ such that
\[
  \cE(\B^\prime) - \cE(\B) \gg_\delta 1.
\]
\end{corollary}
\begin{proof}
By the previous corollary there is a $1$-measurable set $E\subset[N]$ with growth $\ll_\delta 1$ such that
\[
  |\langle f - \E(f|\B), 1_E\rangle | \gg_\delta 1.
\]
Let $\B^\prime$ be the factor generated by $\B$ and $E$. Then $\B^\prime$ is a $1$-factor of complexity $\leq 2M$ and growth $\ll_{M,\delta,\cF} 1$, and since $1_E$ is $\B^\prime$-measurable we have
\[
  |\langle \E(f|\B^\prime) - \E(f|\B), 1_E\rangle| \gg_\delta 1.
\]
Now Cauchy--Schwarz and the Pythagorean theorem imply that
\[
  \cE(\B^\prime) - \cE(\B) = \|\E(f|\B^\prime) - \E(f|\B)\|_2^2 \gg_\delta 1.\qedhere
\]
\end{proof}

We can now deduce a weak form of the regularity lemma, occasionally referred to as the Koopman--von Neumann theorem.

\begin{corollary}[Weak regularity]
Let $\B$ be a $1$-factor of complexity $M$ and growth $\cF$, and let $f:[N]\to[-1,1]$ be a function. Then there exists a refinement $\B^\prime$ of $\B$ of complexity $\ll_{\delta,M} 1$ and growth $\ll_{\delta,M,\cF} 1$ such that
\[
  \| f - \E(f|\B^\prime)\|_{U^2([N])} \leq \delta .
\]
\end{corollary}
\begin{proof}
Repeatedly apply the previous corollary to refine the $1$-factor $\B$. Since $0\leq \cE(\B)\leq 1$, this process must end after $\ll_{\delta} 1$ steps.
\end{proof}

Finally, by iterating \emph{this} result, we deduce full regularity.

\begin{theorem}[The $U^2$ regularity lemma]\label{thm:nonirratareg}
Let $f : [N] \to [0,1]$ be a function, $\cF$ a growth function, and $\eps>0$. Then there is a quantity $M\ll_{\eps,\cF} 1$ and a decomposition
\[
  f = f_{\str} + f_{\sml} + f_{\unf}
\]
of $f$ into functions $f_{\str}, f_{\sml}, f_{\unf} : [N] \to [-1,1]$ such that
\begin{enumerate}
  \item $f_{\str}$ has $1$-complexity at most $M$,
  \item $f_{\sml}$ has $L^2([N])$ norm at most $\eps$,
  \item $f_{\unf}$ has $U^2([N])$ norm at most $1/\cF(M)$,
  \item $f_{\str}$ and $f_{\str} + f_{\sml}$ take values in $[0,1]$.
\end{enumerate}
\end{theorem}
\begin{proof}
Starting with $M_0=1$ and $\B_0 = \{\emptyset,[N]\}$, suppose inductively that $\B_i$ is a $1$-factor of complexity and growth $\ll_{i,M_i,\cF} 1$. Then there is a function $f_{\str}^{(i)}:[N]\to\R$ of $1$-complexity $M_{i+1}\ll_{\eps, i,M_i,\cF} 1$ such that $M_{i+1}\geq M_i$ and
\[
  \|\E(f|\B_i)-f_{\str}^{(i)}\|_2\leq \eps/2.
\]
Moreover, by truncating $f_{\str}^{(i)}$ above and below (which doesn't increase $1$-complexity) we may assume that $f_{\str}^{(i)}:[N]\to[0,1]$. By the previous corollary there is a refinement $\B_{i+1}$ of $\B_i$ of complexity and growth $\ll_{i,M_{i+1},\cF} 1$ such that
\[
  \|f - \E(f|\B_{i+1})\|_{U^2([N])} \leq 1/{\cF(M_{i+1})}.
\]

Note in the end that $M_i\ll_{\eps,i,\cF} 1$, and since $(\cE(\B_i))$ is an increasing sequence in $[0,1]$ there is some $i\ll_\eps 1$ such that
\[
  \cE(\B_{i+1}) - \cE(\B_i) = \|\E(f|\B_{i+1}) - \E(f|\B_i)\|_2^2 \leq \eps^2/4.
\]
Let $M = M_{i+1}$ and let
\begin{align*}
f_{\str} &= f_{\str}^{(i)}, \\
f_{\sml} &= \E(f|\B_{i+1}) - f_{\str}^{(i)}, \\
f_{\unf} &= f-\E(f|\B_{i+1}). \qedhere
\end{align*}
\end{proof}

It is often convenient to make the structure of $f_{\str}$ a little more explicit. Specifically, we know that $f_{\str}(n) = F(\theta n)$ for some $F:\T^d \to [0,1]$ and $\theta\in\T^d$ such that $d, \|F\|_{\Lip} \leq M$, but what exactly this entails about the behaviour of $f_{\str}$ depends critically on the Diophantine properties of $\theta$. For counting purposes we would like $\theta\in\T^d$ to be \emph{$(A,N)$-irrational} for some large $A$, meaning that if $q\in\Z^d\setminus \{0\}$ and $\|q\|_1 \leq A$ (where if $q = (q_1,\dots,q_d)$ then $\|q\|_1 = |q_1| + \cdots + |q_d|$) then $\|q\cdot\theta\|_\T \geq A/N$: this would guarantee that $\theta n$ rapidly equidistributes over $\T^d$. Of course there are other possible behaviours of $\theta$: it may be that $\theta$ itself is small, in which case $q\cdot\theta$ moves slowly away from $0$, or it may be that $\theta$ is rational, in which case $q\cdot\theta$ frequently returns to $0$, or there may be a combination of these behaviours. Nevertheless, it turns out that once these two pollutants are boiled off, the remnant is highly irrational in the above sense.

We say a subtorus $T$ of $\T^d$ of dimension $d^\prime$ has \emph{complexity} at most $M$ if there is some $L\in\SL_d(\Z)$, all of whose coefficients have size at most $M$, such that $L(T) = \T^{d^\prime}\times\{0\}^{d-d^\prime}$. In this case we implicitly identify $T$ with $\T^{d^\prime}$ using $L$. For instance, we say $\theta\in\T^d$ is \emph{$(A,N)$-irrational in $T$} if $L(\theta)$ is $(A,N)$-irrational in $\T^{d^\prime}$.

\begin{theorem}
Given $\theta\in \T^d$, a positive integer $N$, and a growth function $\cF$, there is a quantity $M\ll_{d,\cF} 1$ and a decomposition
\[
  \theta = \theta_{\smth} + \theta_{\rat} + \theta_{\irrat}
\]
such that
\begin{enumerate}
  \item $\theta_{\smth}$ is $(M,N)$-smooth, meaning $d(\theta_{\smth},0) \leq \frac{M}{N}$,
  \item $\theta_{\rat}$ is $M$-rational, meaning $q\theta_{\rat}=0$ for some $q\leq M$, and
  \item $\theta_{\irrat}$ is $(\cF(M),N)$-irrational in a subtorus of complexity $\leq M$.
\end{enumerate}
\end{theorem} 
\begin{proof}
Starting with $M_0=1$, $\theta_{\smth}^{(0)} = \theta_{\rat}^{(0)} = 0$, $\theta_{\irrat}^{(0)} = \theta$, and $T_0 = \T^d$, suppose inductively that
\[
  \theta = \theta_{\smth}^{(i)} + \theta_{\rat}^{(i)} + \theta_{\irrat}^{(i)},
\]
where $\theta_{\smth}^{(i)}$ is $(M_i,N)$-smooth, $\theta_{\rat}^{(i)}$ is $M_i$-rational, and $\theta_{\irrat}^{(i)}$ lies in a subtorus $T_i$ of dimension $d-i$ and complexity $\leq M_i$.

If $\theta_{\irrat}^{(i)}$ is $(\cF(M_i),N)$-irrational in $T_i$ then we are done, so suppose that $L\in\SL_d(\Z)$ is a linear map of complexity $\leq M_i$ identifying $T$ with $\T^{d-i}$ and such that
\[
  \|q\cdot L(\theta_{\irrat}^{(i)})\|_\T \leq \frac{\cF(M_i)}{N}
\]
for some $q\in\Z^{d-i}\backslash\{0\}$ such that $\|q\|_1\leq\cF(M_i)$. Choose $\theta_{\smth}^{(i)\prime}\in T$ so that
\[
  \|q\cdot L(\theta_{\irrat}^{(i)} - \theta_{\smth}^{(i)\prime})\|_\T = 0
\]
and such that $d(L(\theta_{\smth}^{(i)\prime}),0) \leq \cF(M_i)/N$, so $d(\theta_{\smth}^{(i)\prime},0)\ll_{M_i, d,\cF} 1/N$. Let $q=mq^\prime$ where $m\in\Z^+$ and $q^\prime$ is primitive in $\Z^{d-i}$. Then
\[
  q^\prime\cdot L(\theta_{\irrat}^{(i)}-\theta_{\smth}^{(i)\prime})\in{\textstyle\frac{1}{m}}\Z.
\]
Now using the Euclidean algorithm, choose $\theta_{\rat}^{(i)\prime}\in T$ so that
\[
  q^\prime\cdot L(\theta_{\irrat}^{(i)}-\theta_{\smth}^{(i)\prime}-\theta_{\rat}^{(i)\prime})\in\Z
\]
and such that $m L(\theta_{\rat}^{(i)\prime}) = 0$, so that $m\theta_{\rat}^{(i)\prime} = 0$. Finally, let
\begin{align*}
  \theta_{\smth}^{(i+1)} &= \theta_{\smth}^{(i)} + \theta_{\smth}^{(i)\prime},\\
  \theta_{\rat}^{(i+1)} &= \theta_{\rat}^{(i)} + \theta_{\rat}^{(i)\prime},\\
  \theta_{\irrat}^{(i+1)} &= \theta_{\irrat}^{(i)} - \theta_{\smth}^{(i)\prime} - \theta_{\rat}^{(i)\prime},
\end{align*}
and choose $M_{i+1}\ll_{M_i,d,\cF} 1$ so that $\theta_{\smth}^{(i+1)}$ is $(M_{i+1},N)$-smooth, $\theta_{\rat}^{(i+1)}$ is $M_{i+1}$-rational, and the subtorus $T_{i+1} = \{x\in T_i : q^\prime\cdot L(x) = 0\}$ has complexity $\leq M_{i+1}$.

In the end note that $M_i\ll_{i, d,\cF} 1$, and since $T_i$ has dimension $d-i$ we can iterate this argument no more than $d$ times, so for some $i\leq d$ we must have that $\theta_{\irrat}^{(i)}$ is $(\cF(M_i),N)$-irrational in $T_i$.
\end{proof}

We can now state and prove the irrational version of the regularity lemma. This version improves on Theorem \ref{thm:nonirratareg} by giving $f_{\str}$ the structure
\[
  f_{\str}(n) = F(n/N,n\bmod q, \theta n),
\]
where
\[
  F:[0,1]\times\Z/q\Z\times\T^d\to\R,
\]
$q, d, \|F\|_{\Lip} \leq M$, and $\theta$ is $(\cF(M),N)$-irrational. Here we take the usual Euclidean metrics on $[0,1]$ and $\T^d$, the discrete metric on $\Z/q\Z$, the sum of these metrics on $[0,1]\times\Z/q\Z\times\T^d$, and then define $\|F\|_{\Lip}$ as before.

\begin{theorem}[$U^2$ regularity, irrational version]\label{thm:areg}
Let $f : [N] \to [0,1]$ be a function, $\cF$ a growth function, and $\eps>0$. Then there is a quantity $M\ll_{\eps,\cF} 1$ and a decomposition
\[
  f = f_{\str} + f_{\sml} + f_{\unf}
\]
of $f$ into functions $f_{\str}, f_{\sml}, f_{\unf} : [N] \to [-1,1]$ such that
\begin{enumerate}
  \item $f_{\str}(n) = F(n/N,n\bmod q, \theta n)$, where
\[
  F:[0,1]\times\Z/q\Z\times\T^d\to[0,1],
\]
$q,d,\|F\|_{\Lip} \leq M$, and $\theta\in\T^d$ is $(\cF(M),N)$-irrational,
  \item $f_{\sml}$ has $L^2([N])$ norm at most $\eps$,
  \item $f_{\unf}$ has $U^2([N])$ norm at most $1/\cF(M)$,
  \item $f_{\str}$ and $f_{\str} + f_{\sml}$ take values in $[0,1]$.
\end{enumerate}
\end{theorem}

\begin{proof}
Let $\cF_1$ and $\cF_2$ be growth functions depending on $\eps$ and $\cF$ in a manner to be determined. By Theorem~\ref{thm:nonirratareg} there exists $M_1\ll_{\eps,\cF_1} 1$ and a decomposition
\[
  f = f_{\str} + f_{\sml} + f_{\unf}
\]
of $f$ into functions $f_{\str}, f_{\sml}, f_{\unf} : [N] \to [-1,1]$ such that
\begin{enumerate}
  \item $f_{\str}(n) = F(\theta n)$, where $F:\T^d\to[0,1]$, $d,\|F\|_{\Lip} \leq M_1$, and $\theta\in\T^d$,
  \item $f_{\sml}$ has $L^2([N])$ norm at most $\eps$,
  \item $f_{\unf}$ has $U^2([N])$ norm at most $1/\cF_1(M_1)$, and
  \item $f_{\str}$ and $f_{\str} + f_{\sml}$ take values in $[0,1]$.
\end{enumerate}
Now by the previous theorem we can find $M_2\ll_{M_1,\cF_2} 1$ such that $M_2\geq M_1$ and such that $\theta$ decomposes as
\[
  \theta = \theta_{\smth} + \theta_{\rat} + \theta_{\irrat},
\]
where
\begin{enumerate}
  \item $\theta_{\smth}$ is $(M_2,N)$-smooth, meaning $d(\theta_{\smth},0) \leq \frac{M_2}{N}$,
  \item $\theta_{\rat}$ is $M_2$-rational, meaning $q\theta_{\rat}=0$ for some $q\leq M_2$, and
  \item $\theta_{\irrat}$ is $(\cF_2(M_2),N)$-irrational in a subtorus of complexity $\leq M_2$.
\end{enumerate}
Then
\[
  F(\theta n) = F(\theta_{\smth} n + \theta_{\rat} n + \theta_{\irrat} n) = \tilde F(n/N,n\bmod q,nL(\theta_{\irrat})),
\]
where $\tilde F:[0,1]\times\Z/q\Z\times\T^{d^\prime}\to[0,1]$ is defined by
\[
  \tilde F(x,y,z) = F(N\theta_{\smth}x + \theta_{\rat}y + L^{-1}(z)).
\]

Noting that $\|\tilde F\|_{\Lip} \ll_{M_2} 1$, we can find $M\ll_{M_2} 1$ exceeding both $M_2$ and $\|\tilde F\|_{\Lip}$. But since $M\ll_{M_2} 1$, if $\cF_2$ is sufficiently large depending on $\cF$ then $\cF_2(M_2)\geq\cF(M)$, and similarly $M_2\ll_{M_1,\cF_2} 1$, so if $\cF_1$ is sufficiently large depending on $\cF_2$ then $\cF_1(M_1)\geq \cF_2(M_2)\geq \cF(M)$. After all these dependencies are fixed we have $M\ll_{\eps,\cF} 1$, and the conclusion of the theorem holds.
\end{proof}

In applications one typically combines the arithmetic regularity lemma with some sort of counting lemma such as the following. As mentioned already, if $\theta \in \T^d$ is highly irrational (i.e., $(A,N)$-irrational for large $A$), then the sequence $\theta n$ is highly equidistibuted over $\T^d$ as $n$ ranges over long progressions. This allows us to relate counts weighted by $f_{\str}$ to integrals of $F$.

\begin{lemma}\label{distribution-integral-a}
Suppose that $\theta \in \T^d$ is $(A, N)$-irrational, and let $F : \T^d \to \C$ be a function with Lipschitz constant at most $M$. Suppose that $P \subset \{1,\dots, N\}$ is an arithmetic progression of length at least $\eta N$. Then, provided that $A > A_0(M, d, \eta, \delta)$ is large enough,
\[
  \left|\E_{n \in P} F(\theta n) - \int F \, d\mu \right| \leq \delta.
\]
\end{lemma}
\begin{proof}
The key here (as usual in equidistribution theory) is to take a Fourier expansion of $F$ and truncate it. In particular, we may find $M_0 = O_{M, d, \delta}(1)$ and coefficients $c_m$ with $c_0 = \int F$ and $c_m = O_{M,d}(1)$ for $m\neq 0$ such that
\[
  \left| F(x) - \sum_{\|m\|_1 \leq M_0} c_m e(m \cdot x) \right| \leq \delta/2
\]
uniformly in $x$. For a proof, see for example \cite[Lemma A.9]{green-tao-quadraticuniformity}. It follows, of course, that 
\[
  \left| \E_{n \in P} F(\theta n) - \int F\, d\mu \right| \leq  \sum_{\|m\|_1 \leq M_0, m \neq 0} |c_m| | \E_{n \in P} e(m \cdot \theta n)| + \frac{\delta}{2}.
\]
Thus we need only show that 
\[
  \E_{n \in P} e(m \cdot \theta n) = o_{m, \eta; A \rightarrow \infty}(1),
\]
and then take $A$ sufficiently large. If the common difference of the arithmetic progression $P$ is $h$, then by summing the geometric progression we have the bound
\[
  \E_{n \in P} e(m \cdot \theta n) \ll \frac{1}{\eta N \| (m \cdot \theta) h \|_\T}.
\]
But if $A > |h|\|m\|_1$ then, by the definition of $(A,N)$-irrationality, $\| (m \cdot \theta) h \|_\T \geq A/N$. Since $h \leq 2\eta^{-1}$, the result follows immediately. 
\end{proof}

If $f_{\str}$ has the structure given by Theorem~\ref{thm:areg} and $\cF$ grows sufficiently rapidly, then the triple $(n/N, n\bmod q, n\theta)$ is highly equidistributed over $[0,1]\times \Z/q\Z \times \T^d$ as $n$ ranges over $\{1,\dots,N\}$. Thus we have the following slightly more involved counting lemma, which is proved in essentially the same way.

\begin{lemma}\label{distribution-integral}
Suppose that $\theta \in \T^d$ is $(A, N)$-irrational. Let $q \in \N$, and let $F : [0,1]\times \Z/q\Z \times \T^d \rightarrow \C$ be a function with Lipschitz constant at most $M$. Let $\delta > 0$ be arbitrary. Then, provided that $A > A_0(M, q, d, \delta)$ and $N>N_0(M,q,d,\delta)$ are large enough,
\[ 
  \left| \E_{n \leq N} F(n/N, n \bmod{q}, \theta n) - \int F \, d\mu \right| \leq \delta.
\]
\end{lemma}
\begin{proof}[Proof sketch]
Again the idea is to take a truncated Fourier expansion of $F$, but because $F|_{\{0\}\times \Z/q\Z \times \T^d}$ and $F|_{\{1\}\times \Z/q\Z \times \T^d}$ need not agree the expansion looks a little more complicated. However, $F$ can be extended to an $M$-Lipschitz function $[-1,1]\times \Z/q\Z\times \T^d\to\C$ such that $F(x,y,z)=F(-x,y,z)$, so $F$ may be approximated by a sum of the functions $\phi_{k,a,m}$ given by
\begin{equation}\label{phi}
  \phi_{k,a,m}(x,y,z) = e\(\frac{m}{2}x + \frac{a}{q}y + m \cdot z\) + e\(- \frac{m}{2} x + \frac{a}{q}y + m \cdot z\),
\end{equation}
where $k\in\Z, a \in \Z/q\Z, m \in \Z^d$. Then just as in the proof of the previous lemma we need only check that 
\begin{equation}\label{to-check-22}
  \E_{n \leq N} \phi_{k,a, m} (n/N, n \bmod q, \theta n) = o_{k, a, m, q; A,N \to \infty}(1)
\end{equation}
provided that $k,a,m$ are not all zero. Substituting in, the left-hand side is 
\begin{equation}\label{to-check-3}
  \E_{n \leq N} \(e\(\(\frac{k}{2N} + \frac{a}{q} + m \cdot \theta\) n  \) + e\(\(- \frac{k}{2N} + \frac{a}{q} + m \cdot \theta\) n  \)\).
\end{equation}
Summing the geometric progressions, we see that this is bounded by $\eps$ unless
\begin{equation}\label{to-use-9}
  \left\|\pm \frac{k}{2N} + \frac{a}{q} + m \cdot \theta \right\|_\T \ll \frac{1}{N\eps}
\end{equation}
for either choice of sign.

Supposing first that $m \neq 0$, inequality \eqref{to-use-9} implies
\[
  \left\| \pm\frac{mk}{2N} + q m \cdot \theta \right\|_\T \ll \frac{q}{N\eps},
\]
and hence
\[
  \left \| m' \cdot \theta \right \|_\T \ll \frac{q}{N\eps} + \frac{mq}{2N},
\]
where $m' = qm$. But if $A$ is sufficiently large in terms of $\eps, q,k,m$, this is contrary to the $(A,N)$-irrationality of $\theta$.

Hence suppose that $m = 0$. Then if $N$ is large enough depending on $m$ and $q$, \eqref{to-use-9} implies that $a = 0$. Thus $a = m = 0$, so $k \neq 0$. But then the expression \eqref{to-check-3} is
\[
  \E_{n \leq N}\(e(kn/2N) + e(-kn/2N)\) = O_k(1/N),
\]
so \eqref{to-check-22} certainly follows in this case as well.
\end{proof}

\bibliography{arithreg}

\begin{thebibliography}{EGM14}

\bibitem[EGM14]{EGM}
S.~Eberhard, B.~Green, and F.~Manners.
\newblock Sets of integers with no large sum-free subset.
\newblock {\em Ann. of Math. (2)}, 180(2):621--652, 2014.

\bibitem[GT08]{green-tao-quadraticuniformity}
B.~Green and T.~Tao.
\newblock Quadratic uniformity of the {M}\"obius function.
\newblock {\em Ann. Inst. Fourier (Grenoble)}, 58(6):1863--1935, 2008.

\bibitem[GT10]{areg}
B.~Green and T.~Tao.
\newblock An arithmetic regularity lemma, an associated counting lemma, and
  applications.
\newblock In {\em An irregular mind}, volume~21 of {\em Bolyai Soc. Math.
  Stud.}, pages 261--334. J\'anos Bolyai Math. Soc., Budapest, 2010.

\bibitem[Rud87]{rudin-real-complex}
W.~Rudin.
\newblock {\em Real and complex analysis}.
\newblock McGraw-Hill Book Co., New York, third edition, 1987.

\bibitem[Tao12]{higher-fourier}
T.~Tao.
\newblock {\em Higher order Fourier analysis}.
\newblock Graduate Studies in Mathematics. American Mathematical Society, 2012.

\end{thebibliography}
\bibliographystyle{alpha}

\end{document}